\documentclass{amsart}
\usepackage{amssymb}
\theoremstyle{plain}
\newtheorem{lem}{Lemma}[section]
\newtheorem{thm}{Theorem}[section]
\newtheorem{cor}{Corollary}

\theoremstyle{definition}
\newtheorem{remark}{Remark}
\newtheorem{dfn}{Definition}[section]
\numberwithin{equation}{section}
\input amssym.def
\input amssym
\newcommand{\D}{{\mathbb D}}

\newcommand{\B}{{\mathbb B}}
\newcommand{\Ss}{{\mathbb S}}

\newcommand{\R}{{\mathbb R}}

\begin{document}

\author{Kari Astala and Vesna Manojlovi\'c}
\address{Le Studium, Loire Valley Institute for Advanced Studies, Orleans \& Tours, France, MAPMO University of Orleans and Department of Mathematics, University of Helsinki, Helsinki, Finland; Faculty of Organizational Sciences, Univ. of Belgrade, Mathematical Institute, SASA, Belgrade,
Serbia}
\email{\rm kari.astala@helsinki.fi, vesnam@fon.bg.ac.rs}
\thanks{
K.A.~was supported by the Academy of Finland (SA) grant 12719831. V.M.~was supported by Ministry of Science, Serbia, project OI174017.}

\title{On Pavlovic's theorem in space}
\maketitle

\begin{abstract}
We study higher dimensional counterparts to the well-known theorem of Pavlovic \cite{pa3},  that   every harmonic quasiconformal mapping of the disk is bi-Lipschitz.  
\end{abstract}

\section{Introduction}

In his influential paper \cite{pa3} Pavlovic  showed that harmonic quasiconformal mappings of the unit disk $\D$ onto itself are bi-Lipschitz mappings. The paper has %has been very influential and 
initiated an extensive investigation  between the Lipschitz conditions and  harmonic quasiconformal mappings, see e.g.  \cite{abm}, \cite{bm}, \cite{k2}, \cite{k4}, \cite{kap},Ê %\cite{km}, \cite{kp}, 
\cite{mv} and their references. %[Improve/expand the list !]

In this paper we study  counterparts of Pavlovic's theorem in higher dimensions. 

\begin{thm} \label{main}
 Suppose $f:\B^3 \to \B^3$ is a harmonic quasiconformal mapping, which is also %. Then % $f$ is Lipschitz-continuous. Iif, in addition, $f$ is
   a gradient mapping,
that is $f = \nabla u$ for some  function $u$ harmonic in the unit ball $\B^3$. Then $f$ is a bi-Lipschitz mapping.
\end{thm}

In two dimensions  Pavlovic  made a deep and detailed analysis of the boundary values of $f$; analysing them he achieved  the Lipschitz-property for every harmonic quasiconformal mapping of the disk.
In higher dimensions Pavlovic's approach seems difficult to work with; instead it would seem conceivable that the Lipschitz-property follows by the regularity theory of elliptic PDE's. In fact, such an approach was done by Kalaj \cite{k4}. However, the proof in \cite{k4} is rather long and technical, and one of the purposes of this note is to give 
%Therefore for readers convenience we give 
a simple  and self-contained argument showing the Lipschitz property in all dimensions.
%, which will also  lead to generalizations of the results in \cite{k4} (?).

Thus the  main difficulty is to find lower bounds for  $|f(x) - f(y)|$ in terms of the distance between $x$ and $y$. In general dimensions it is not even known if  harmonic quasiconformal mappings of the ball have non-vanishing Jacobian. On the other hand, in three dimension Lewy \cite{lewy} proved that for homeomorphic harmonic gradient mappings the Jacobian determinant has no zeroes, and building on this together with work of Gleason and Wolff \cite{gw}  one  arrives at Theorem \ref{main}. 

\section{Lipschitz properties in higher dimensions}

We start with Lipschitz properties for harmonic quasiconformal mappings of the ball, and consider Lipschitz bounds in more general domains  in subsequent sections. %In the case of the unit ball we  present a quick argument. 

\begin{thm} \label{3.1}
If $n \geq 2$  and $f : \B^n \to \B^n$ is a harmonic and $K$-quasiconformal mapping, then 
$$ |f(x) - f(y)| \leq L|x-y|, \qquad x,y \in \B^n,
$$
where  $L$ depends only on the distortion $K$,  dimension $n$ and  ${ \rm dist}(f(0),\Ss^{n-1})$.
\end{thm} 

For the proof of Theorem \ref{3.1} we only need the Sobolev embedding, which we use in the following local form.

\begin{lem} \label{sobo}
Suppose that $w \in W^{2,1}_{loc}(\B^n) \cap C( \overline{\B^n}\, )$, that $h \in L^p(\B^n)$ for some $1 < p < \infty$ and that
$$ \Delta w = h \; \mbox{ in }  \B^n, \mbox{ with } w\big|_{ \Ss^{n-1}} = 0,
$$

a) If $1 < p < n$, then 
$$ \| \nabla w \|_{L^q(\B^n)} \leq c(p,n) \| h \|_{L^p(\B^n)}, \qquad \frac{1}{q} = \frac{1}{p} - \frac{1}{n}. 
$$

b) If $n < p < \infty$, then
$$  \| \nabla w \|_{L^{\infty}(\B^n)}  \leq c(p,n) \| h  \|_{L^p(\B^n)}. 
$$
%b) If $n < p < \infty$, then
%$$  \| \nabla w \|_{L^{\infty}(\B^n)} + \| \nabla w \|_{Lip_{\alpha}} \leq c(p,n) \| h  \|_{L^p(\B^n)},
% \qquad  \alpha = 1 - \frac{n}{p}.
%$$
\end{lem}

The standard proof of Lemma \ref{sobo} %, see e.g. \cite{?}, \cite{?}, 
follows from the fact that one can represent $w$ in terms of the Green's function
$G_{\B^n}(x,y)$ of the unit ball,
$$w(x) = \int_{\B^n} G_{\B^n}(x,y) h(y) dm(y), \qquad x \in \B^n.
$$
The Green's function and its gradient
$$ \nabla_x G_{\B^n}(x,y) =  c_1(n)\left[ \frac{x-y}{|x-y|^n} + |y|^n \frac{y - |y|^2 x}{\big| y- |y|^2 x  \big|^n}\right]
$$
can be explicitly calculated. Since $|y| |x-y| \leq |y - |y|^2x|$ for all $x,y \in \B^n$, the gradient is bounded by  $$| \nabla_x G_{\B^n}(x,y)| \leq 2 c_1(n) |y-x|^{1-n}\quad  \mbox{for} \quad x,y \in \B^n.$$
Therefore $ \| \nabla w \|_{L^q(\B^n)} \leq c  \| {\mathcal I}_1 h \|_{L^q(\R^n)}$, where ${\mathcal I}_s h$ denotes the Riesz potential of order $s$.  Thus   Lemma \ref{sobo}.$a)$ reduces to the well known boundedness properties of the Riesz potentials,
$$ \| {\mathcal I}_s h \|_{L^q(\R^n)} \leq c(s,p,q) \|  h \|_{L^p(\R^n)}, \qquad \frac{1}{q} = \frac{1}{p} - \frac{s}{n},$$
 given  e.g. in \cite[p.]{stein}. % for the bound $a)$.
  The bound in  $ b)$ is easier and follows from H\"older's inequality, since $y \mapsto |x-y|^{1-n} \in L^q(\B^n)$ for every $1 \leq q < \frac{n}{n-1}$.
\medskip

The above local form of Sobolev's embedding yields a quick proof for the following. 
%generalisation of \cite[Theorem B]{k4}.
\smallskip

\begin{cor} \label{ite}
Suppose $w \in W^{2,1}_{loc}(\B^n) \cap C( \overline{\B^n}\, )$, $n \geq 2$, is such that 
\begin{equation} \label{let}
w\big|_{ \Ss^{n-1}} = 0, \quad \mbox{with} \quad 
\int_{\B^n} |\nabla w|^{p_0} \; dm  < \infty \; \; \mbox{for some }  n < p_0 < \infty.
\end{equation}
If $w$ satisfies the following uniform differential inequality,
\begin{equation} \label{bound}
 |\Delta w (x) | \leq a |\nabla w(x) |^2 + b, \qquad x \in \B^n,
\end{equation}
for some constants $\, a, b < \infty$, we then have
\begin{equation} \label{claim}
 \| \nabla w \|_{L^\infty(\B^n)}  \leq M < \infty,  %\qquad 0 < \alpha < 1, \| w\big|_{\Ss^{n-1}} \|_{\infty} +
\end{equation}
where $M = M(a,b, p_0,  \|Ê\nabla w \|_{p_0})$. 
%\begin{equation} \label{claim}
 %\| \nabla w \|_{L^\infty(\B^n)} + \|w \|_{C^{1,\alpha}(\B^n)} \leq M < \infty, \qquad 0 < \alpha < 1, 
%\end{equation}
%where $M = M(a,b, p, \| w \|_{\infty} + \|Ê\nabla w \|_p, \alpha)$. 
In particular, $w$ is Lipschitz continuous.
\end{cor}
\smallskip

\begin{proof} 

%Note first that we may assume that  $w$ vanishes on the boundary. Namely, if  $\phi = w\big|_{ \Ss^{n-1}}$, let $w_0$ %solve  the  Dirichlet problem in $\B^n$ with these  (continuous by assumption) boundary values.   
%From gradient bounds for the Poisson kernel one gets \cite[Theorem ?.?]{?}

According to  \eqref{bound} we have 
\begin{equation} \label{bound3}
\Delta w(x) = h(x), \qquad  \mbox{for  a.e. } x \in \B^n,
\end{equation}
 where
\begin{equation} \label{bound2}
h(x) = c(x) \left(  |\nabla w (x) |^2 + 1\right) 
\end{equation}
and  $ \|c \|_{\infty} \leq \max\{a,b\}$; one can simply define 
$c(x) :=  \Delta w (x) \left(  |\nabla w (x) |^2 + 1\right)^{-1}$ for almost every $x \in \B^n$.

Here by our assumptions  $\nabla w \in L^{p_0}(\B^n)$. However, with Sobolev embedding one can improve this integrability, up to 
\begin{equation}
\label{raja}
\nabla w \in L^{s}(\B^n) \qquad \mbox{where } s > 2n.
\end{equation}  
Indeed, if $p_0/2 < n < p_0$, then %by \eqref{bound2}, 
$h(x) = c(x) \left(  |\nabla w (x) |^2 + 1\right)  \in L^{p_0/2}(\B^n)$, and \eqref{bound3} with Lemma \ref{sobo} $a)$ give
\begin{equation}
\label{ }
\nabla w \in L^{p_1}(\B^n),  \qquad   p_1 = \frac{p_0 n}{2n-p_0} > p_0,
\end{equation} which is
a strict improvement in the integrability. 

To quantify this, note that if initially $p_0 = n(1+\varepsilon)$, $\varepsilon > 0$, then
$$ p_1 =  \frac{p_0 n}{2n-p_0} = n  \frac{1+\varepsilon}{1-\varepsilon} > n(1+2\varepsilon)
$$
Thus one can iterate this feedback argument, getting $\nabla w \in L^{p_\ell}(\B^n)$, $\ell = 0,1,2, \dots$ with
$p_{\ell} > n(1+2^{\ell} \varepsilon)$, until the condition \eqref{raja} is reached. (If it happens that for some exponent $p_\ell = 2 n$, we can choose $p_0$ little smaller so that this degeneracy does not happen.)

And once \eqref{raja} is achieved, \eqref{bound3}-\eqref{bound2}  with Sobolev embedding, Lemma \ref{sobo} b), give $\| \nabla w \|_\infty < \infty$.
The proof also gives a bound for $ \| \nabla w \|_\infty$ that depends only on the constants $a$ and $b$, the exponent $p_0$ and the initial norm $ \| h \|_{p_0/2} \leq \max\{a,b\}(\|Ê\nabla w \|_{p_0}^2 +1)$.
\end{proof}

\begin{remark} It is interesting to note that the above iteration argument fails if in  \eqref{let} one assumes integrability only for some $1 \leq p_0 < n$. Thus higher integrability, and Gehring's theorem \cite{geh} in case of quasiconformal mappings, become particularly useful also here.\end{remark}

\begin{remark} One can replace the zero boundary values in \eqref{let} e.g. by the requirement $w\big|_{\Ss^{n-1}} \in C^{1,\alpha}$, by considering $w - P[w]$, where $P[w]$ is the Poisson integral of $w$. Similarly, by properties of the Green's function the conclusions can be improved to $ \| \nabla w \|_{L^\infty(\B^n)} + \|w \|_{C^{1,\alpha}(\B^n)} \leq M < \infty, \qquad 0 < \alpha < 1 $% [Check !!!!!!!!!!!!!!!!!!!!!!!!!!]
\end{remark}

We can now turn to proving the Lipschitz bounds for  harmonic $K$-quasiconformal mappings $f = (f^1,\dots,f^n): \B^n \to \B^n$. For this note that by harmonicity
$$ \Delta (f^j)^2(x) = 2 |\nabla f^j(x)|^2, \qquad j=1,\dots,n.
$$
Thus  any "reasonable" function of $f^1, \dots, f^n $ will satisfy the differential inequality \eqref{bound}. To get uniform  Lipschitz bounds, we  need  in addition some normalisation such as vanishing on the boundary $ \Ss^{n-1}$, like in  \eqref{let}. Therefore a convenient choice for our purposes is  e.g. $w(x) = 1-|f(x)|^2$.

\medskip

\noindent {\it Proof of Theorem \ref{3.1}}. We first recall  Gehrings famous  theorem \cite{geh} which gives for every quasiconformal mapping  $f: \R^n \to \R^n$ the higher integrability
\begin{equation}\label{geh2}
 \int_{\B^n} |Df(x)|^p dm \leq C < \infty, \qquad p = p(n,K) > n,% \bigl( { \rm dist}(f(0),\Ss^{n-1}), K,p \bigr).  %\qquad [check \; normalization] 
\end{equation}
where for mappings of the whole space $\R^n$, the constant $C$ depends only on $n$ and distortion $K(f)$.

In case $f:\B^n \to \B^n$ is $K$-quasiconformal, we can compose $f$ with a M\"obius transform $\psi$ preserving the ball, such that $f \circ \psi (0) = 0$. With  Schwarz reflection one can then extend $f \circ \psi$ to $\R^n$ and apply \eqref{geh2} to this mapping. Unwinding the M\"obius transform, i.e. after a change of variables, we see that any $K$-quasiconformal mapping $f:\B^n \to \B^n$ satisfies \eqref{geh2} with $C = C(n,K, { \rm dist}(f^{-1}(0),\Ss^{n-1}) \bigr)$

If in addition $f$ is harmonic, consider the function
$$ w(x) = 1-|f(x)|^2, \qquad x \in \B^n.
$$
Since quasiconformal mappings of $\B^n$ extend continuously  to the boundary, $w(x)$ satisfies the assumptions of Corollary \ref{ite}. For the condition \eqref{bound} note that $w = u \circ f$ where
$$ u(x) = 1- |x|^2 \quad \mbox{with} \quad \nabla u(x) = - 2 x, \qquad x \in \B^n.
$$
Thus $\nabla w(x) = Df^t(x) \nabla u(f(x))$ so that 
\begin{equation}\label{kolme}
 \frac{2}{K} |f(x)| \,|Df(x)| \leq |\nabla w(x)| \leq  2|f(x)|\, |Df(x)|
\end{equation}
with
$$ |\Delta w(x)| =  2||Df(x)||^2 \leq 2n^2 |Df(x)|^2, \qquad x \in \B^n.
$$
where $||Df(x)||^2$ denotes the Hilbert-Schmidt norm of the differential matrix.

The above already establishes \eqref{bound}. However, to see the explicit dependence of $a$ and $b$ on properties of the mapping $f$ we first note that there is a constant $\delta = \delta\bigl(n,K,a, { \rm dist}(f(0),\Ss^{n-1})\bigr)$ such that
\begin{equation}
\label{sharp}
1-|x| + |f(x)| \geq \delta > 0, \qquad \mbox{for all } x \in \B^n.
\end{equation}
 Indeed, as quasiconformal mappings  of $\B^n$Êare rough isometries in the hyperbolic metric \cite{vuo},  if $M := h_{\B^n}\bigl(0,f(0)\bigr) $, then either $ h_{\B^n}\bigl(f(x),f(0)\bigr) \geq 2M \Rightarrow |f(x)| \geq \frac{e^M - 1}{e^M +1}$, or else $ h_{\B^n}(x,0) \leq c(K)(2M+1) \Rightarrow 1-|x| \geq e^{-c(K)(2M+1)}$. Thus \eqref{sharp} holds, and  we have
 %$\delta := \min\{ \delta_1,\delta_2 \}$ satisfies \eqref{sharp}, and  we have
 $$ |\Delta w| \leq \frac{4n^2}{\delta^2} [ (1-|x|)^2   + |f(x)|^2 ]  |Df(x)|^2 \leq \frac{2Kn^2}{\delta^2} |\nabla w|^2 + \frac{4n^2}{\delta^2} (1-|x|)^2 |Df(x)|^2.$$
The last term is controlled by basic ellipticity bounds \cite[p.38]{gt}, i.e. the Bloch norm bounds
\begin{equation}\label{bloch}
 (1-|x|)|Df(x)|  \leq c(n) \|f\|_\infty
\end{equation}
valid for every harmonic function. Thus \eqref{bound} holds with $a = 4n^2 \delta^{-2}$, $b =  4n^2 c(n)^2\delta^{-2}$, 
 so that $\nabla w \in L^\infty(\B^n)$ by Corollary \ref{ite}. A combination of \eqref{kolme}-\eqref{bloch} shows finally that $f$ is a Lipschitz mapping, with Lipschitz constant $L\leq (c(n) +K \|Ê\nabla w\|_{\infty}/2)/\delta$.
\hfill $\Box$

\section{Co-Lipschitz Mappings}

We say that a mapping $f$ defined in a domain $\Omega \subset \R^n$ has the co-Lipschitz property with constant $1 \leq L$, if 
\begin{equation} \label{colip}
 |f(x) - f(y)| \geq \frac{1}{L}|x-y| \qquad \forall \;  x, y \in \Omega.
\end{equation} 
The inverse of a $K$-quasiconformal mapping is also $K$-quasiconformal mapping, but   for harmonic $f$ the inverse $f^{-1}$ is not in general harmonic. Hence even for harmonic quasiconformal mappings of the ball, the
 co-Lipschitz property does not follow from Theorem \ref{3.1}. 
 
Naturally, for mappings in \eqref{colip} the Jacobians are non-vanishing everywhere. In dimensions $n \geq 3$, the Jacobian of a  harmonic homeomorphism may vanish,  see e.g. \cite[p.26]{duren}, and therefore the co-Lipschitz property is a  more subtle problem than in dimension $n=2$. 

On the other hand, for quasiconformal mapping we have the following geometric notion of an average derivative,  see
 \cite[Definition 1.5]{ag},
\begin{equation} \label{ave}
\alpha_f(z)=\exp\left(\frac 1n(\log{J_f})_{B_z}\right).
\end{equation} 
Here
$$
(\log J_f)_{B_z}=\frac 1{m(B_z)}\int_{B_z}\log J_f\,dm,\quad B_z=B(z,d(z,\partial \Omega)).
$$
Since for a quasiconformal mapping, the Jacobian $J_f$ is an $A_\infty$-weight, $\alpha_f(z)$ is comparable to
$\left( \frac 1{m(B_z)}\int_{B_z} J_f^p \right)^{1/p}$ for every $0 < p \leq 1$, and hence we could have used such averages, as well. On the other hand, in the case $n=2$ and $f$ conformal we have
$$
\alpha_f(z)= |f'(z)|
$$
and therefore the above choice \eqref{ave} appears a natural  one. Furthermore, we have the following quasiconformal version of the Koebe-distortion theorem, see \cite[Theorem 1.8]{ag}.

%Continuity properties of quasiconformal mappings $f:D\longrightarrow D'$, where $D$ and $D'$ are domains in plane,
%with respect to various natural metrics have been studied extensively in \cite{akm}, \cite{km}, \cite{kp} and \cite{pa3}.

%We are going to use the following result
\begin{thm}
\label{agt}
 Suppose that $\Omega$ and $\Omega'$ are domains in $\mathbb R^n$ if $f:\Omega\longrightarrow \Omega'$ is
$K$-qc, then
$$
\frac 1c\,\frac{d(f(z),\partial \Omega')}{d(z,\partial \Omega)}\leq\alpha_f(z)\leq c\,\frac{d(f(z),\partial \Omega')}{d(z,\partial \Omega)}
$$
for $z\in \Omega$, where $c$ is a constant which depends only on $K$ and $n$.
\end{thm}

As a  step towards the co-Lipschitz properties of harmonic quasiconformal mappings we  prove the following general lower bound.

\begin{thm}
\label{avebdd}
Suppose $f:\B^n\longrightarrow \Omega$ is a harmonic quasiconformal mapping, with $\Omega \subset \R^n$ a convex subdomain. Then 
\begin{equation}
\label{ }
\alpha_f(x) \geq c_0 \, d(f(0),\partial \Omega)> 0, \qquad x \in \B^n,
\end{equation}
where the constant $c_0 = c_0(n, K) $ depends only on the dimension $n$ and distortion $K = K(f)$.
\end{thm}
\begin{proof}
For every $z\in \B^n$ we have
$$
d(f(z),\partial \Omega)=\inf_pd(f(z),p),
$$
where infimum is taken over all lines $p$ outside domain. Since $d(f(z),p) = \; \break <f(z),n> + \; const.$, where $n$ is a normal to $p$, the function $z \mapsto d(f(z),p)$ is positive and harmonic in $\B^n$. We denote this function by $h_p(z)$,  and for each $h_p$ apply the usual Harnack inequality in $\B^n$,
$$
h_p(z)\geq \frac{1-|z|}{(1+|z|)^{n-1}}\,h_p(0).
$$
Because $d(f(0),p)\geq d(f(0),\partial \Omega)$ we have
$$
h_p(z)\geq \frac{1-|z|}{(1+|z|)^{n-1}}\,d(f(0),\partial \Omega).
$$
Infimum of the last inequality over all $p$ gives
$$
d(f(z),\partial \Omega)\geq \frac{1-|z|}{(1+|z|)^{n-1}}\,d(f(0),\partial \Omega).
$$
Finally, as
$$
d(z,\partial \B^n)=1-|z|
$$
the last inequality we can write as
$$
\frac{d(f(z),\partial \Omega)}{d(z,\partial \B^n)}\geq \frac{d(f(0),\partial \Omega)}{(1+|z|)^{n-1}}.
$$
Using then Theorem \ref{agt} and quasiconformality of $f$ we conclude that
$$
\alpha_f(z)\geq c(n,K) d(f(0),\partial \Omega).
$$
\end{proof}

Thus one can achieve the co-Lipschitz property if the usual derivative can be estimated from below by the average derivative.
In two dimensions this can be done by the next key result of the second author, see \cite{man}.

\begin{thm} \label{vesnam}
Suppose $\Omega, \Omega' \subset \R^2$ are planar domains and $f:\Omega \to \Omega' $  a harmonic quasiconformal mapping. Then $\log J_f$ is superharmonic in $\Omega$.
\end{thm} 

Now, we can use the superharmonicity of  $\log J_f$  %is superharmonic in $\Omega$, we have
for the harmonic quasiconformal mapping $f$ defined in the unit disk $\B^2$,
\begin{equation}
\label{super}
\log |Df(x)|^2 \geq \log J_f(z) \geq \frac 1{m(B_z)}\int_{B_z}\log J_f\,dm = \log \alpha_f(z)^2, \qquad z \in \B^2.
 \end{equation}
This estimate combined with Theorem \ref{avebdd} proves for every harmonic quasiconformal mapping from the disk onto a convex domain the lower bound
 \begin{equation}
\label{super2} 
\inf_{|h|=1}|Df(x)h| \geq |Df(x)|/K \geq \alpha_f(x)/K \geq c d(f(0),\partial \Omega)
 \end{equation}
  for some constant $c>0$. 
From this we can conclude that $f$ is co-Lipschitz. This is new proof of theorem \cite[Cor 2.7]{k5}. In fact we have more generally, 
\begin{cor} \label{cor2} Suppose $\Omega, \Omega' \subset \R^2$ are simply connected domains and $f:\Omega \to \Omega' $  is a harmonic quasiconformal mapping. If $\Omega'$ is convex and the Riemann map of $\Omega$ has derivative bounded from above, the $f$ has the coLipschitz property \eqref{colip}.
 \end{cor}

The proof follows by applying \eqref{super2} to $f \circ g$, where $g:\D \to \Omega$ is the Riemann map. 
So in particular, in  Corollary \ref{cor2}Ê  the boundary of $\Omega$ need not be  $C^1$, not even Lipschitz. For instance, $g(z) = 2z-z^2$ is a conformal map from $\D$ onto a cardioid, with  cusp at $1 = g(1)$.

Similarly, combining \eqref{super} with Theorems \ref{3.1} and \ref{avebdd} we have a new proof for Pavlovic's theorem in $\B^2$.
\medskip

To complete the proof of Theorem \ref{main} we use an argument analogous to \eqref{super} and Theorem \ref{vesnam}. 
First, by Lewy's theorem \cite{lewy}, if the gradient $f= \nabla u$  of a (real valued) harmonic function defines a homeomorphism $f:\Omega \to \Omega' $ where $\Omega, \Omega' \subset \R^3$, then the Jacobian $J_f$ does not vanish. Further, $J_f = \mathcal H_u$, where $\mathcal H_u$ denotes the  Hessian of $u$, and here we have the  theorem of Gleason and Wolff \cite[Theorem A]{gw} that  again in dimension three, the function \, $\log \det(\mathcal H_u)$\,   is superharmonic outside the zeroes of the Hessian. We collect these facts in the following

\begin{thm}[Lewy-Gleason-Wolff] \label{lgw} Suppose % $\Omega \subset \R^3$ is a domain and
 $u:\Omega \to \R$ is a harmonic function, such that $f(x) := \nabla u(x)$ defines a homeomorphism between the domains $\Omega$ and $\Omega' \subset \R^3$.
Then 
\begin{equation}
\label{super5}
\log J_f(z) =  \log \det(\mathcal H_u) \mbox{ is superharmonic in } \Omega.
\end{equation}
\end{thm}

With these arguments %arguing as in \eqref{super} 
we finally have a proof of Theorem \ref{main}. Indeed, if $f= \nabla u$ is a quasiconformal harmonic gradient mapping in $\B^3$ then as in \eqref{super}, using Theorem \ref{lgw} we conclude that 
$$\alpha_f(x)^3 \leq J_f(x) \leq  K(f)^2 \inf_{|h|=1}|Df(x)h|^3.$$ Thus when $f(\B^3) = \B^3$, or more generally when the target domain is convex, Theorem \ref{avebdd} gives the co-Lipschitz property for $f$. The Lipschitz-properties follow from Theorem \ref{3.1}, completing the proof of Theorem \ref{main}. \hfill $\Box$

\section{Harmonic quasiconformal mappings in general domains}

The above approach to higher dimensional Pavlovic's theorem  has a few consequences also in more general subdomains of $\R^n$.
To discuss these, we start with the quasihyperbolic metric introduced by Gehring and Palka in \cite{gp}. In this work they used the metric as a tool to understand  quasiconformal homogeneity. The metric has since been studied by number of different authors. 

\begin{dfn}
Let $D$ be a proper subdomain of the $\mathbb R^n$, $n\geq 2$. We define quasihyperbolic length
of a rectifiable curve $\gamma\subset D$ by
$$
l_k(\gamma)=\int_\gamma\frac{ds}{d(x,\partial D)}.
$$
The quasihyperbolic metric is defined by
$$
k_D(x_1,x_2)=\inf_\gamma(l_k(\gamma)),
$$
where the infimum is taken over all rectifiable curves in $D$ joining $x_1$ and $x_2$.
\end{dfn}

Quasihyperbolic metric is invariant under Euclidean isometries and homoteties but it is not invariant under
conformal mappings, it is  not even M\" obius invariant. By result of Gehring and Osgood \cite{go},
for any domain $D\subseteq\mathbb R^n$ and points $x,y\in D$ there exists a quasihyperbolic geodesic.
Moreover, quasihyperbolic metric is quasi-invariant under conformal and more generally under quasiconformal mappings.
Namely, there is a constant $0 < C = C(n,K) < \infty$ such that
$$
k_{D'}(f(x_1),f(x_2))\leq C\cdot\max\{k_D(x_1,x_2),k_D^\alpha(x_1,x_2)\},\quad \alpha=K^{1/(1-n)},
$$
for all $x_1,x_2\in D$, whenever $f$ is a $K$-quasiconformal mapping from $D$ onto $D'$.

If we deal with harmonic quasiconformal mappings between two general proper planar domains,  then such mappings are bi-Lipschitz with respect to corresponding quasihyperbolic metrics \cite{man}.
Here we have a generalization of this result in space.

\begin{thm} Consider domains $\Omega, \Omega' \subset \mathbb R^3$, and let $f:\Omega \to \Omega'$ be a harmonic quasiconformal homeomorphism  which is also a gradient mapping, $f = \nabla u$ for some  function $u$ harmonic in
$\Omega$. Then $f$ is bi-Lipschitz with respect to the corresponding quasihyperbolic metrics,
$$
\frac{1}{M} k_{\Omega}(x,y) \leq k_{\Omega'}(f(x),f(y))  \leq M k_{\Omega}(x,y), \qquad x, y \in \Omega, 
$$
where the constant $M$ depends only on the distortion $K(f)$.
\end{thm}

\begin{proof} From \eqref{super5} and Theorem \ref{lgw} we get $\alpha_f(x)^3 \leq J_f(x)$, $x \in \Omega$. % \leq \| Df(x) \|^3$.
 On the other hand, $Df(x)$ is a vector valued harmonic function, whose norm is subharmonic and thus
$$  \| Df(x) \|^3 \leq \frac 1{m(B_x)}\int_{B_x}  \| Df \|^3\,dm  \leq \frac K{m(B_x)}\int_{B_x}  J_f \,dm $$
$$ \leq C(K,n) \exp [\frac 1{m(B_x)}\int_{B_x}\log J_f\,dm] = C(K,n) \alpha_f(x)^3,
 $$
 where the third inequality follows from the fact that $J_f$ is an $A_{\infty}$-Muckenhoupt weight. Thus 
 $ \alpha_f(x) \simeq \inf_{|h|=1}|Df(x)h| \simeq \sup_{|h|=1}|Df(x)h|$, and the claim follows as in \cite{man}.
\end{proof}

The proof of Theorem \ref{main} at the end of  the previous section gives immediately 

\begin{cor} Suppose $f: \B^3 \to \Omega$ is  quasiconformal. If $\Omega$ is convex and $f = \nabla u$ is the gradient of  a harmonic function, then $f$ has the co-Lipschitz property \eqref{colip}.
 \end{cor}
 
 Similarly, method of Theorem 2.1. works for more general domains. We have the following result of Kalaj \cite{k4}.
 
 \begin{cor} If $f: \B^n \to \Omega$ is  a harmonic quasiconformal mapping, where $\Omega \subset \R^n$ is a domain with $C^2$-boundary, then $f$ is a Lipschitz mapping.
  \end{cor}
 \begin{proof}
We take this time $w(x) = {\rm dist}(f(x), \partial \Omega)$ near $ \partial \Omega$, and choose some smooth extension to $\Omega$. Then $w$ satisfies the inequality \eqref{bound}, see \cite{k4}, so that $\|Ê\nabla w \|_{\infty} < \infty$ by Corollary \ref{ite} and we obtain the Lipschitz bounds for $f$ as in the proof of Theorem \ref{3.1}.
\end{proof}

Collecting the above information we also have

\begin{thm} 
Suppose  \, $\Omega$ is a convex subdomain of $\R^3$ with $C^2$-boundary, and let $f: \B^3 \to \Omega$ be a harmonic quasiconformal homeomorphism. If $f = \nabla u$ is a harmonic gradient mapping, the $f$ is  bi-Lipschitz.
\end{thm}

\end{document}